\documentclass[12pt,reqno]{amsart}
\usepackage{amscd,amssymb,amsthm}
\usepackage{graphicx}
\usepackage{enumerate}
\theoremstyle{plain}
\newtheorem{theorem}{Theorem}[section]
\newtheorem{corollary}[theorem]{Corollary}
\newtheorem{lemma}[theorem]{Lemma}

\newcommand{\floor}[1]{\left\lfloor{#1}\right\rfloor}
\newcommand{\reel}{\mathbb{R}}
\newcommand{\comp}{\mathbb{C}}
\newcommand{\ds}{\displaystyle}
\newcommand{\Log}{{\rm Log}}
\newcommand{\Arg}{{\rm Arg}}
\newcommand{\abs}[1]{\left\vert #1\right\vert }

\title[The Sum of Certain Series Related to Harmonic Numbers ]
{The Sum of Certain Series Related to Harmonic Numbers}
\author{Omran Kouba}
\address{Department of Mathematics \\
Higher Institute for Applied Sciences and Technology\\
P.O. Box 31983, Damascus, Syria.}
\email{omran\_kouba@hiast.edu.sy}

\begin{document}
\date{\today}
\begin{abstract}
In this paper, we consider three families of numerical series with general terms containing 
the harmonic numbers, and we use simple methods from
classical and complex analysis to find explicit formul\ae  \ for their respective sums.
\end{abstract}
\smallskip\goodbreak

\maketitle

\section{\bf Introduction }\label{sec1}
\parindent=0pt
\quad Let $H_n=\sum_{j=1}^n1/j$ be the $n$th harmonic number. For a positive integer
$k$, let $S_k$, $T_k$ and $U_k$ denote, respectively, the sum of the following series :
\begin{align*}
S_k&=\sum_{n=1}^\infty(-1)^{n-1}(\log k-(H_{kn}-H_n)),\\
T_k&=\sum_{n=1}^\infty\frac{\log k-(H_{kn}-H_n)}{n},\\
U_k&=\sum_{n=1}^\infty(-1)^{n-1}\frac{H_{kn}}{n}.
\end{align*}

\quad The question of finding the value of  $T_k$  was asked by Ovidiu Furdui in \cite{fur}, and was answered by the present author\cite{kou1}. While evaluating
$S_k$ was the object of a problem posed by the author \cite{kou2}. 

\smallskip

\quad In this paper we will present a unified approach to determine these sums. It consists of finding an integral representation of each one of these sums, and then calculating the corresponding integral.

\smallskip

\quad The paper is organized as follows. In section~\ref{sec2}, we gathered some preliminary lemmas. Lemma \ref{lem21} and its corollaries are of interest in their own right. In section~\ref{sec3}, we find the statements and proofs of the main theorems.

\section{\bf Preleminaries }\label{sec2}

\quad In our first lemma, we prove that a certain complex function satisfies a simple funtional equation. This is the main tool in the 
proof of our results. Namely, Theorem \ref{th32} and Theorem \ref{th33}.

\begin{lemma}\label{lem21}
Let $\Omega=\comp\setminus[0,+\infty[$ that is the set of complex numbers 
with a cut along the set of nonnegative real numbers. For $z$ in $\Omega$ we define
$F(z)$ by
$$F(z)=\int_0^1\frac{\log(1-t)}{z-t}\, d t.$$
 Then $F$ satisfies the following functional equation
\begin{equation*}\label{E:dag}
\forall\, z\in\Omega,\quad F(z)+F\left(\frac1z\right)=\frac{\pi^2}{6}-\Log(1-z)~\Log\left(1-\frac1z\right)\tag{$\dag$}
\end{equation*}
where $\Log$ is the principal branch of the logarithm.
\end{lemma}

\begin{proof}
Note first that both $F$ and $z\mapsto\Log(1-z)$ are holomorphic in the connected region $\Omega$, and since $\Omega$ is invariant under the holomorphic mapping $z\mapsto 1/z$, we conclude that
$$z\mapsto G(z)=
F(z)+F\left(\frac1z\right)+\Log(1-z)~\Log\left(1-\frac1z\right)
$$
is holomorphic in $\Omega$, so to prove the lemma, it is sufficient to prove that $G(x)=\pi^2/6$ for each negative real $x$. See [1, Ch.4, \$3.]\par
\quad Now for $x\in(-\infty,0)$ we have, (using integration by parts)
\begin{align*}
F^\prime(x)&=-\int_0^1\frac{\log(1-t)}{(t-x)^2}\, d t\\
&=\left[\left(\frac{1}{t-x}-\frac{1}{1-x}\right)\log(1-t)\right]_{t=0}^{t=1}
+\frac{1}{1-x}\int_0^1\frac{ d t}{t-x}\\
&=\frac{1}{1-x}\log\left(1-\frac{1}{x}\right),
\end{align*}
and one checks immediatly that $G^\prime(x)=0$ for every negative real $x$. This proves that, for some constant $c$, we have $G(x)=c$
for every $ x$ in the interval $(-\infty,0)$.   Letting $x$ tend to $0^-$, (and noting that $\lim_{x\to-\infty}F(x)=0$,) we conclude that 
\begin{align*}
c&=F(0)=\int_0^1\frac{-\log(1-t)}{t}\, d t
=\int_0^1\sum_{n=1}^\infty\frac{t^{n-1}}{n}\, d t\\
&=\sum_{n=1}^\infty\frac{1}{n}\int_0^1t^{n-1}\, d t=\sum_{n=1}^\infty\frac{1}{n^2}=\frac{\pi^2}{6}.
\end{align*}
This concludes the proof of the lemma.
\end{proof}

\quad Our first corollary is a formula, ``\`a la BBP  \cite{bbp}'', for $\pi^2$,  that allows the direct computation of binary of hexadecimal digits
of  $\pi^2$. See also \cite{exp}.

\begin{corollary}\label{cor21} If $(a_1,a_2,a_3,a_4,a_4,a_6,a_7)=(16,-16,-8,-16,-4,-4,2)$, then
$$\pi^2=\sum_{n=0}^\infty\frac{1}{16^n}\left(\sum_{r=1}^7\frac{a_r}{(8n+r)^2}\right).$$
\end{corollary}

\begin{proof}
Indeed, putting $z=-1$ in \eqref{E:dag} we find that
\begin{align*}
\frac{\pi^2}{12}-\frac{\log^22}{2}&=\int_0^1\frac{\log(1-t)}{-1-t}\,dt
=\frac{1}{2}\int_0^1\frac{-\log u}{1-u/2}\,du\\
&=\sum_{k=0}^\infty\frac{1}{2^{k+1}}\int_0^1u^k(-\log u)\, du=\sum_{k=1}^\infty\frac{1}{2^{k}k^2}
\end{align*}
Separating the above series into four series according to the value of $r=k\mod4$, we obtain
\begin{equation*}\label{E:eq1}
\frac{\pi^2}{12}-\frac{\log^22}{2}=\sum_{n=0}^\infty\frac{1}{2^{4n}}\left(\sum_{r=1}^4\frac{2^{2-r}}{(8n+2r)^2}\right).\tag{1}
\end{equation*}

Similarly, taking $z=i$ in \eqref{E:dag} we find that
\begin{align*}
\frac{\pi^2}{12}-\frac{1}{2}\abs{\Log(1+i)}^2&=\Re\left(\int_0^1\frac{\log(1-t)}{i-t}\,dt\right)
=\Re\left(\int_0^1\frac{-\log u}{1-i-u}\,du\right)\\
&=\Re\left(\sum_{k=0}^\infty\frac{1}{(1-i)^{k+1}}\int_0^1u^k(-\log u)\, du\right)=\sum_{k=1}^\infty\frac{\Re((1+i)^k)}{2^{k}k^2}
\end{align*}
That is
$$\frac{\pi^2}{12}-\frac{1}{2}\left(\frac{\log^22}{4}+\frac{\pi^2}{16}\right)
=\sum_{k=1}^\infty\frac{\cos(k\pi/4)}{2^{k/2}k^2}$$
and again, separating the above series according to the value of $r=k\mod8$, we obtain
\begin{equation*}\label{E:eq2}
\frac{5\pi^2}{96}-\frac{\log^22}{8}
=\sum_{n=0}^\infty\frac{1}{2^{4n}}\left(\sum_{r=1}^8\frac{2^{-r/2}\cos(r\pi/4)}{(8n+r)^2}\right).\tag{2}
\end{equation*}
And the desired formula follows by adding $32$ times \eqref{E:eq2} to $-8$ times \eqref{E:eq1}.
\end{proof}

\quad As we have seen before, we can use Lemma \ref{lem21} to evaluate many integrals, the following two corollaries illustrate this. Other applications are Theorem \ref{th32} and Theorem \ref{th33}.

\begin{corollary}\label{cor22} 
Let $\alpha$ be a real number from $[-1,1)$. Then 
$$\int_0^1\frac{(\alpha-t)\log(1-t)}{1-2\alpha t+t^2}\, dt
=\frac{\pi^2}{12}-\frac{(\arccos(\alpha)-\pi)^2}{8}-\frac18\log^2\left(2(1-\alpha)\right).$$
\end{corollary}
\smallskip
\begin{proof}
 Let $\theta=\arccos(\alpha)\in(0,\pi]$. Using
\eqref{E:dag}, with $z=e^{i\theta}$ we conclude that
$$2\Re \left(F(e^{i\theta})\right)=\frac{\pi^2}{6}-\abs{\Log(1-e^{i\theta})}^2,$$
but 
\begin{align*}
1-e^{i\theta}= -2i\sin\left(\frac{\theta}{2}\right) e^{i\theta/2}= 2\sin\left(\frac{\theta}{2}\right) e^{i(\theta-\pi)/2},
\end{align*}
hence
\begin{align*}
\Log(1-e^{i\theta})&=\log\left(2\sin\left(\frac{\theta}{2}\right)\right)+i\frac{\theta-\pi}{2}\\
&=\frac12\log\left(2(1-\cos\theta)\right)+i\frac{\theta-\pi}{2}\\
&=\frac12\log\left(2(1-\alpha)\right)+i\frac{\arccos(\alpha)-\pi}{2}.
\end{align*}
\quad On the other hand, we have
\begin{align*}
2\Re \left(F(e^{i\theta})\right)&=\int_0^1\frac{\log(1-t)}{e^{i\theta}-t}\, d t+\int_0^1\frac{\log(1-t)}{e^{-i\theta}-t}\, d t\\
&=\int_0^1\frac{2(\alpha-t)\log(1-t)}{1-2\alpha t+t^2}\, dt.
\end{align*}
It follows that
$$\int_0^1\frac{(\alpha-t)\log(1-t)}{1-2\alpha t+t^2}\, dt
=\frac{\pi^2}{12}-\frac{(\arccos(\alpha)-\pi)^2}{8}-\frac18\log^2\left(2(1-\alpha)\right),$$
which is the desired conclusion.
\end{proof}

{\bf Examples.}  In particular, choosing $\alpha\in\{-1/2,0,1/2\}$, we find that
\begin{align*}
\int_0^1\frac{(1+2t)\log(1-t)}{1+ t+t^2}\, dt
&=-\frac{5\pi^2}{36}+\frac14\log^23,\\
\int_0^1\frac{t\log(1-t)}{1+t^2}\, dt
&=-\frac{5\pi^2}{96}+\frac18\log^22,\\
\int_0^1\frac{(1-2t)\log(1-t)}{1-t+t^2}\, dt
&=\frac{\pi^2}{18}.
\end{align*}

\smallskip
\quad The following corollary is a generalization of Corollary \ref{cor22}.

\begin{corollary}\label{cor23} 
Let $P(X)$ be a real polynomial of degree $n$, and let $\{a_1,\ldots ,a_n\}$ be the roots of $P$, each one is repeated according to its multiplicity.
  Assume that the roots of $P$ belong to ${\mathcal U}'=\{z\in\comp:\abs{z}=1, z\ne1\}$. Then
$$\int_0^1\frac{P'(t)}{P(t)}\log(1-t)\, dt
=-\frac{n\pi^2}{12}+\frac12\sum_{j=1}^n\log^2\abs{1-a_j}+\frac12\sum_{j=1}^n(\Arg(1-a_j))^2,$$
where $\Arg$ is the principal determination of the argument, {\it i.e.} the one that belongs to $(-\pi,\pi)$.
\end{corollary}
\smallskip
\begin{proof}
Indeed, since $P$ is real we have $P(X)=\overline{P(X)}$. Hence, there is a nonzero real $\lambda$ such that
$$P(X)=\lambda\prod_{j=1}^n(X-a_j)=\lambda\prod_{j=1}^n\left(X-1/a_j\right).$$
Therefore,
$$\frac{P'(X)}{P(X)}=\sum_{j=1}^n\frac{1}{X-a_j}=\sum_{j=1}^n\frac{1}{X-1/a_j}.$$
It follows that
\begin{align*}
2\int_0^1\frac{P'(t)}{P(t)}\log(1-t)\, dt&=\int_0^1\left(\sum_{j=1}^n\frac{\log(1-t)}{t-a_j}+\sum_{j=1}^n\frac{\log(1-t)}{t-1/a_j}\right)\, dt\\
&=-\sum_{j=1}^n\left(\int_0^1\frac{\log(1-t)}{a_j-t}\, dt+\int_0^1\frac{\log(1-t)}{(1/a_j)-t}\, dt\right)\\
&=-\sum_{j=1}^n\left(F(a_j)+F\left(\frac{1}{a_j}\right)\right)\\
&=-\sum_{j=1}^n\left(\frac{\pi^2}{6}-\Log(1-a_j)\Log(1-1/a_j)\right)\\
&=-\frac{n\pi^2}{6}+\sum_{j=1}^n\Log(1-a_j)\Log(1-\overline{a_j})\\
&=-\frac{n\pi^2}{6}+\sum_{j=1}^n\abs{\Log(1-a_j)}^2\\
&=-\frac{n\pi^2}{6}+\sum_{j=1}^n\log^2\abs{1-a_j}+\sum_{j=1}^n(\Arg(1-a_j))^2
\end{align*}
This concludes the proof of the corollary.
\end{proof}

\quad We invite the reader to discover other applications of Lemma \ref{lem21}. In the next lemma we find an integral representation of  the quantity  $\log k-(H_{kn}-H_n)$, and this will help us in the task of summing the series under consideration.

\begin{lemma}\label{lem24}
Let $n$ and $k$ be integers such that $n\geq 1$ and $k\geq2$. Then
\begin{equation*}
\log k-(H_{kn}-H_n)=\int_0^1\frac{Q'_k(t)}{Q_k(t)}t^{nk}\, dt,
\end{equation*}
where $Q_k$ is the polynomial $Q_k(t)=1+t+\cdots+t^{k-1}$.
\end{lemma}
\begin{proof}

Since $(1-t)Q_k(t)=1-t^k$ we have  $(1-t)Q^\prime_k(t)=Q_k(t)-kt^{k-1}$, and consequently, for $n\geq1$ and $t\in(0,1)$, 
\begin{align*}
\frac{Q^\prime_k(t)}{Q_k(t)}(1-t^{nk})
&=(1-t)Q^\prime_k(t)\frac{1-t^{nk}}{1-t^k}\\
&=\frac{1-t^{nk}}{1-t^k}(Q_k-kt^{k-1}),
\end{align*}
that is
\begin{align*}
\frac{Q^\prime_k(t)}{Q_k(t)}(1-t^{nk})& =\left(1+t^k+t^{2k}+\cdots+t^{(n-1)k}\right)\left(
1+t+t^{2}+\cdots+t^{k-1}-kt^{k-1}\right)\\
&=\sum_{j=1}^{nk}t^{j-1}-k\sum_{\ell=1}^{n}t^{k\ell-1}.
\end{align*}
We conclude, that
\begin{align*}
\int_0^1\frac{Q^\prime_k(t)}{Q_k(t)}\,t^{nk} d t
&=\int_0^1\frac{Q^\prime_k(t)}{Q_k(t)}\, d t-\sum_{j=1}^{nk}\int_0^1t^{j-1} d t+k
\sum_{\ell=1}^{n}\int_0^1t^{k\ell-1} d t\\
&=\log k-H_{kn}+H_n.
\end{align*}
This ends the proof of the lemma.
\end{proof}

\quad In the next lemma we find an integral representation of  the quantity  $H_n/n$, which is useful is summing many series containing similar
expressions. In particular, it will be used in the proof of Theorem \ref{th33}.
\smallskip

\begin{lemma}\label{lem25}
Let $n$ be an integers such that $n\geq 1$. Then
\begin{equation*}
\frac{H_{n}}{n}=\int_\Delta\frac{y^{n-1}}{1-x}\, dx\,dy
\end{equation*}
where $\Delta=\{(x,y)\in\reel^2:0\leq x<y\leq 1\}$.
\end{lemma}
\begin{proof}
This is easy. Indeed
\begin{align*}
\int_\Delta\frac{y^{n-1}}{1-x}\, dy\, dx, &
=\int_{x=0}^1\frac{1}{1-x}\left(\int_{y=x}^1y^{n-1}\,dy\right)\,dx\\
&=\int_{0}^1\frac{1-x^n}{n(1-x)}\,dx=\frac{1}{n}\int_{0}^1\bigg(\sum_{j=1}^nx^{j-1}\bigg)\,dx\\
&=\frac{1}{n}\sum_{j=1}^n\int_{0}^1x^{j-1}\,dx=\frac{H_n}{n}
\end{align*}
which is the desired conclusion.
\end{proof}

\smallskip

\section{\bf The Main Results }\label{sec3}

\smallskip

\quad The evaluation of the sum of the first of our three series, does not use Lemma \ref{lem21}, so it is the ``easiest'' one.

\smallskip\goodbreak

\begin{theorem}\label{th31} For an integer $k\geq 2$, let $S_k$ be defined by
$$
S_k=\sum_{n=1}^\infty(-1)^{n-1}(\log k-(H_{kn}-H_n)),
$$
then,
$$S_k
=\frac{k-1}{2k}\log2+\frac{1}{2}\log k-
\frac{\pi}{2k^2}\sum_{\ell=1}^{\floor{k/2}}(k+1-2\ell)\cot\left(\frac{(2\ell-1)\pi}{2k}\right).$$
\end{theorem}

\begin{proof}
Using Lemma \ref{lem24} we have
\begin{equation*}
\log k-(H_{kn}-H_n)=\int_0^1\frac{Q'_k(t)}{Q_k(t)}t^{nk}\, dt,
\end{equation*}
where $Q_k$ is the polynomial $Q_k(t)=1+t+\cdots+t^{k-1}$. Now, for $m>1$ we have
$$\sum_{n=1}^{m-1}(-1)^{n-1}(\log k-H_{nk}+H_n)
=\int_0^1\frac{Q^\prime_k(t)}{Q_k(t)}\cdot\frac{t^k-(-t^k)^m}{1+t^k}\, d t,$$
so that
\begin{align*}
\abs{\sum_{n=1}^{m-1}(-1)^{n-1}(\log k-H_{nk}+H_n)
-\int_0^1\frac{Q^\prime_k(t)}{Q_k(t)}\cdot\frac{t^k}{1+t^k}\, d t}
& =\int_0^1\frac{Q^\prime_k(t)}{Q_k(t)}\cdot\frac{t^{km}}{1+t^k}\, d t\\
\leq M_k & \int_0^1t^{km}\, d t=\frac{M_k}{km+1},
\end{align*}
where $M_k=\sup_{t\in[0,1]}\frac{Q^\prime_k(t)}{(1+t^k)Q_k(t)}$.
This proves the convergence of the series defining $S_k$, and proves also that
\begin{equation*}\label{E:eq3}
S_k=\int_0^1\frac{Q^\prime_k(t)}{Q_k(t)}\cdot\frac{t^k}{1+t^k}\, d t.\tag{3}
\end{equation*}
But, we aleady noted that  $(1-t)Q_k(t)=1-t^k$, so
$$\frac{Q^\prime_k(t)}{Q_k(t)}=\frac{1}{1-t}-\frac{kt^{k-1}}{1-t^k},$$
and we can write \eqref{E:eq3} as follows,
\begin{equation*}\label{E:eq4}
S_k=\int_0^1\left(\frac{1}{1-t}-\frac{kt^{k-1}}{1-t^k}\right)\frac{t^k}{1+t^k}\, d t.\tag{4}
\end{equation*}

\smallskip
For $x\in[0,1)$ we have
\begin{align*}
\int_0^x\left(\frac{1}{1-t}-\frac{kt^{k-1}}{1-t^k}\right)\frac{t^k}{1+t^k}\,d t
=&\int_0^x\frac{t^k}{(1-t)(1+t^k)}\, d t-\int_0^x\frac{kt^{2k-1}}{1-t^{2k}}\, d t\\
=&\int_0^x\frac{t^k-1}{2(1-t)(1+t^k)}\, d t+\int_0^x\frac{t^k+1}{2(1-t)(1+t^k)}\, d t\\
&\qquad -\int_0^x\frac{kt^{2k-1}}{1-t^{2k}}\, d t\\
=&-\frac12\int_0^x\frac{Q_k(t)}{1+t^k}\, d t-\frac{\log(1-x)}{2}
+\frac{\log(1-x^{2k})}{2}\\
=&-\frac12\int_0^x\frac{Q_k(t)}{1+t^k}\, d t
+\frac12\log\left(\frac{1-x^{2k}}{1-x}\right),
\end{align*}
so, taking the limit as $x$ approaches 1, we see that \eqref{E:eq4}  can be written as follows
\begin{equation*}\label{E:eq5}
S_k=\frac{\log(2k)}{2}-\frac12\int_0^1\frac{Q_k(t)}{1+t^k}\, d t.\tag{5}
\end{equation*}
Now, if $\omega=\omega_k=\exp(i\pi/k)$ then $1+t^k=\prod_{j=0}^{k-1}(1-\omega^{2j+1}t)$, consequently
$$\frac{Q_k(t)}{1+t^k}=\sum_{j=0}^{k-1}\frac{\lambda_j}{1-\omega^{2j+1}t},$$
with
\begin{align*}
\lambda_j&=\lim_{z\to\overline{\omega}^{2j+1}}\frac{(1-\omega^{2j+1}z)Q_k(z)}{1+z^k}
=\frac{-\omega^{2j+1}Q_k(\overline{\omega}^{2j+1})}{k(\overline{\omega}^{2j+1})^{k-1}}\\
&=\frac{1}{k}Q_k(\overline{\omega}^{2j+1})
=\frac{1}{k}\cdot\frac{1-(\overline{\omega}^{2j+1})^k}{1-\overline{\omega}^{2j+1}}\\
&=\frac{2}{k(1-\overline{\omega}^{2j+1})},
\end{align*}
hence,
$$\frac{Q_k(t)}{1+t^k}=\frac{2}{k}\sum_{j=0}^{k-1}\frac{1}{1-\omega^{2j+1}}\cdot
\frac{-\omega^{2j+1}}{1-\omega^{2j+1}t}.$$
\quad Clearly,
$t\mapsto \Log(1-\omega^{2j+1}t)$ is a primitive of 
$\ds t\mapsto \frac{-\omega^{2j+1}}{1-\omega^{2j+1}t}$ on the interval $[0,1]$, consequently
\begin{equation*}
\int_0^1\frac{Q_k(t)}{1+t^k}\, d t
=\frac{2}{k}\sum_{j=0}^{k-1}\frac{\Log(1-\omega^{2j+1})}{1-\omega^{2j+1}},
\end{equation*}
and since the left side of this formula is real, we conclude that
\begin{equation*}
\int_0^1\frac{Q_k(t)}{1+t^k}\, d t
=\frac{2}{k}\sum_{j=0}^{k-1}\Re\left(\frac{\Log(1-\omega^{2j+1})}{1-\omega^{2j+1}}\right).
\end{equation*}
But, 
\begin{align*}
1-\omega^{2j+1}&=\exp\left(\frac{\pi(2j+1)i}{2k}\right)(-2i)\sin\left(\frac{\pi(2j+1)}{2k}\right)\\
&=2\sin\left(\frac{\pi(2j+1)}{2k}\right)\exp\left(\frac{i\pi}{2}(\frac{2j+1}{k}-1)\right),
\end{align*}
so
\begin{align*}
\Log(1-\omega^{2j+1})&=\log\abs{1-\omega^{2j+1}}+i \frac{\pi}{2}\left(\frac{2j+1}{k}-1\right),\\
\frac{1}{1-\omega^{2j+1}}&=\frac{1}{2}+\frac{i}{2}\cot\left(\frac{(2j+1)\pi}{2k}\right),
\end{align*}
therefore, we can write $(6)$ as follows :
$$\int_0^1\frac{Q_k(t)}{1+t^k}\, d t
=\frac{1}{k}\log\left(\prod_{j=0}^{k-1}\abs{1-\omega^{2j+1}}\right)
-\frac{\pi}{2k^2}\sum_{j=0}^{k-1}(2j+1-k)\cot\left(\frac{(2j+1)\pi}{2k}\right)$$
From $1+t^k=\prod_{j=0}^{k-1}(1-\omega^{2j+1}t)$ we conclude that
$$\prod_{j=0}^{k-1}\abs{1-\omega^{2j+1}}=\abs{\prod_{j=0}^{k-1}(1-\omega^{2j+1})}=2,$$
 so
$$\int_0^1\frac{Q_k(t)}{1+t^k}\, d t
=\frac{\log2}{k}
-\frac{\pi}{2k^2}\sum_{j=0}^{k-1}(2j+1-k)\cot\left(\frac{(2j+1)\pi}{2k}\right).$$
Finally, since replacing $j$ by $k-1-j$ does not change the summand in the above sum, we obtain
$$\int_0^1\frac{Q_k(t)}{1+t^k}\, d t
=\frac{\log2}{k}
+\frac{\pi}{k^2}\sum_{0\leq j<(k-1)/2}(k-1-2j)\cot\left(\frac{(2j+1)\pi}{2k}\right)$$
and the desired conclusion follows from \eqref{E:eq5} :
\begin{equation*}
S_k=\frac{\log(2k)}{2}-\frac{\log2}{2k}
-\frac{\pi}{2k^2}\sum_{0\leq j<(k-1)/2}(k-1-2j)\cot\left(\frac{(2j+1)\pi}{2k}\right),
\end{equation*}
which is equivalent to the statement of the theorem. This concludes the proof.
\end{proof}

{\bf Examples.}  In particular, we have
\begin{align*}
\sum_{n=1}^\infty(-1)^{n-1}(\log2-H_{2n}+H_n)&=\frac{3}{4}\log2-\frac{\pi}{8},\\
\sum_{n=1}^\infty(-1)^{n-1}(\log3-H_{3n}+H_n)&=
\frac{1}{3}\log2+\frac{1}{2}\log 3-\frac{\pi}{3\sqrt{3}},\\
\sum_{n=1}^\infty(-1)^{n-1}(\log4-H_{4n}+H_n)&=
\frac{11}{8}\log2-
(1+2\sqrt2)\frac{\pi}{16}.
\end{align*}
Subtracting  the last one from twice the first, we obtain
$$\sum_{n=1}^\infty(-1)^{n-1}( H_{4n}+H_n-2H_{2n})=\frac{1}{8}\log2-(3-2\sqrt2)\frac{\pi}{16},$$
and this can be rearranged to give
$$\sum_{n=1}^\infty\left(\sum_{k=1}^{2n}\frac{(-1)^{k+n}}{k+2n}\right)=\frac{1}{8}\log2-(3-2\sqrt2)\frac{\pi}{16}.$$

\smallskip

\begin{theorem}\label{th32} For an integer $k\geq 2$, let $T_k$ be defined by
$$
T_k=\sum_{n=1}^\infty\frac{\log k-(H_{kn}-H_n)}{n},
$$
then,
$$T_k=
\frac{(k-1)(k+2)}{24k}\,\pi^2-\frac12\log^2k-\frac12\sum_{j=1}^{k-1}\log^2\left(2\sin\frac{j\pi}{k}\right).$$
\end{theorem}

\begin{proof}
Again, using Lemma \ref{lem24} we have
\begin{equation*}
\log k-(H_{kn}-H_n)=\int_0^1\frac{Q'_k(t)}{Q_k(t)}t^{nk}\, dt,
\end{equation*}
where $Q_k$ is the polynomial $Q_k(t)=1+t+\cdots+t^{k-1}$. The functions of the sequence
$\ds\left( t\mapsto \frac{Q^\prime_k(t)t^{kn}}{nQ_k(t)}\right)_{n\geq1}$, are positive and continuous on $[0,1]$, so
$$
\int_0^1\frac{Q^\prime_k(t)}{Q_k(t)}\left(\sum_{n=1}^\infty\frac{t^{kn}}{n}\right)\, d t
=\sum_{n=1}^\infty\frac1n\int_0^1\frac{Q^\prime_k(t)}{Q_k(t)}t^{kn}\, d t,$$
that is
$$-\int_0^1\frac{Q^\prime_k(t)}{Q_k(t)}\log(1-t^k)\, d t=\sum_{n=1}^\infty\frac{\log k-(H_{kn}-H_n)}{n}=T_k.$$
Now, $\log(1-t^k)=\log(1-t)+\log Q_k(t)$, so
\begin{align*}
T_k&= -\int_0^1\frac{Q^\prime_k(t)}{Q_k(t)}\log(1-t)\, d t
-\int_0^1\frac{Q^\prime_k(t)}{Q_k(t)}\log (Q_k(t))\, d t\\
&=-\int_0^1\frac{Q^\prime_k(t)}{Q_k(t)}\log(1-t)\, d t-\left[\frac12\log^2Q_k(t)\right]_{t=0}^{t=1}.
\end{align*}
Finally, 
\begin{equation*}\label{E:eq6}
T_k=-\frac12\log^2k+J_k\quad \hbox{with}\quad J_k=-\int_0^1\frac{Q^\prime_k(t)}{Q_k(t)}\log(1-t)\, d t.\tag{6}
\end{equation*}

\smallskip\goodbreak

Now, let $\omega$ denote the $k$th root of unity : $\exp\left(\frac{2i\pi}{k}\right)$. Since $Q_k$ 
is a real polynomial of degree $k-1$ whose roots are $\{\omega^j:0< j<k\}$, 
the evaluation the integral $J_k$ can be done using Corollary \ref{cor23}, as follows : 

$$J_k  =\frac{(k-1)\pi^2}{12}-\frac12\sum_{j=1}^{k-1}\log^2\abs{1-\omega^j}
-\frac12\sum_{j=1}^{k-1}\Arg^2(1-\omega^j).$$
But, for $1\leq j< k$ we have 
$$1-\omega^j=2 \sin\left(\frac{j\pi}{k}\right)\cdot e^{i\left(\frac{ j\pi}{k}-\frac{\pi}{2}\right)}$$
consequently $\abs{1-\omega^j}=2 \sin\left(\frac{j\pi}{k}\right)$ and $\Arg(1-\omega^j)=\left(\frac{j}{k}-\frac12\right)\pi$. 
Therefore,
$$J_k=\frac{(k-1)\pi^2}{12}-\frac12\sum_{j=1}^{k-1}\log^2\left(2\sin\frac{j\pi}{k}\right)
-\frac{\pi^2}{2}\sum_{j=1}^{k-1}\left(\frac{j}{k}-\frac12\right)^2,$$

but
\begin{align*}
\sum_{j=1}^{k-1}\left(\frac{j}{k}-\frac12\right)^2&=
\frac{1}{k^2}\cdot\frac{(k-1)k(2k-1)}{6}-\frac{1}{k}\cdot\frac{(k-1)k}{2}+\frac{k-1}{4}\\
&=\frac{(k-1)(k-2)}{12k},
\end{align*}
hence,
\begin{equation*}\label{E:eq7}
J_k=\frac{(k-1)(k+2)\pi^2}{24k}-\frac12\sum_{j=1}^{k-1}\log^2\left(2\sin\frac{j\pi}{k}\right).\tag{7}
\end{equation*}
Clearly, the conclusion of the theorem follows from \eqref{E:eq6} and \eqref{E:eq7}.
\end{proof}

\smallskip

{\bf Examples.}  In particular,
\begin{align*}
\sum_{n=1}^\infty \frac{\log 2-(H_{2n}-H_n)}{n}&=\frac{1}{12}\pi^2-\log^22,\\
\sum_{n=1}^\infty \frac{\log 3-(H_{3n}-H_n)}{n}&=\frac{5}{36}\pi^2-\frac{3}{4}\log^23,\\
\sum_{n=1}^\infty \frac{\log 4-(H_{4n}-H_n)}{n}&=\frac{3}{16}\pi^2-\frac{11}{16}\log^24,\\
\sum_{n=1}^\infty \frac{\log 5-(H_{5n}-H_n)}{n}&=\frac{7}{30}\pi^2-\frac{5}{8}\log^25-\frac{1}{2}\log^2\left(\frac{1+\sqrt5}{2}\right),\\
\sum_{n=1}^\infty \frac{\log 6-(H_{6n}-H_n)}{n}&= \frac{5}{18}\pi^2-\frac{1}{2}\log^26-\frac{1}{4}\log^23-\frac{1}{2}\log^22.
\end{align*}

\smallskip

\begin{theorem}\label{th33} For an integer $k\geq 1$, let $U_k$ be defined by
$$
U_k=\sum_{n=1}^\infty(-1)^{n-1}\frac{H_{kn}}{n},
$$
then,
$$U_k=
\frac{(k^2+1)\pi^2}{24k}-\frac12\sum_{j=0}^{k-1}\log^2\left(2\sin\frac{(2j+1)\pi}{2k}\right).$$

\end{theorem}

\begin{proof}
Using Lemma \ref{lem25} we have
\begin{equation*}
\frac{H_{kn}}{n}=\int_\Delta\frac{ky^{kn-1}}{1-x}\, dx\,dy=\int_\Delta\frac{ky^{k-1}}{1-x}\, y^{k(n-1)}\, dx\,dy
\end{equation*}
where $\Delta=\{(x,y)\in\reel^2:0\leq x<y\leq 1\}$. Hence, for $m>1$  we have 
\begin{align*}
\sum_{n=1}^m(-1)^{n-1}\frac{H_{kn}}{n}
&=\int_\Delta\frac{ky^{k-1}}{1-x}\left(\sum_{n=1}^m(-y^k)^{n-1}\right)\, dx\,dy\\
&=\int_\Delta\frac{ky^{k-1}}{1-x}\cdot\frac{1-(-y^k)^m}{1+y^k}\, dx\,dy\\
&=\int_\Delta\frac{1}{1-x}\cdot\frac{ky^{k-1}}{1+y^k}\, dx\,dy+(-1)^mR_m,
\end{align*}
where
$$R_m=\int_\Delta\frac{ky^{k(m+1)-1}}{(1-x)(1+y^k)}\, dx\,dy$$
But,
$$0<R_m<\int_\Delta\frac{ky^{k(m+1)-1}}{1-x}\, dx\,dy=\frac{H_{k(m+1)}}{m+1},$$
therefore, $\ds\lim_{m\to\infty}R_m=0$. So, letting $m$ tend to $\infty$, we conclude that
\begin{equation*}
\lim_{m\to\infty} \sum_{n=1}^m(-1)^{n-1}\frac{H_{kn}}{n}=
\int_\Delta\frac{1}{1-x}\cdot\frac{ky^{k-1}}{1+y^k}\, dx\,dy,
\end{equation*}
and we arrive to the following conclusion :

\begin{align*}
U_k&= \int_\Delta\frac{ky^{k-1}}{(1-x)(1+y^k)}\, dx\,dy\\
&=\int_{y=0}^1\frac{ky^{k-1}}{(1-x)(1+y^k)}\left(\int_{x=0}^y\frac{dx}{1-x}\right)dy\\
&=-\int_0^1\frac{ky^{k-1}}{1+y^k}\log(1-y)\,dy. 
\end{align*}

\smallskip\goodbreak

Now, let $\omega=\exp(\frac{i\pi}{k})$ . Since $X^k+1$ 
is a real polynomial of degree $k$ whose roots are $\{\omega^{2j+1}:0\leq j<k\}$, 
the evaluation of $U_k$ can be done using Corollary \ref{cor23}, as follows : 

$$U_k  =\frac{k\pi^2}{12}-\frac12\sum_{j=0}^{k-1}\log^2\abs{1-\omega^{2j+1}}
-\frac12\sum_{j=0}^{k-1}\Arg^2(1-\omega^{2j+1}).$$
But, for $0\leq j< k$ we have 
$$1-\omega^{2j+1}=2 \sin\left(\frac{(2j+1)\pi}{2k}\right)\cdot e^{i(2j+1-k)\pi/(2k)}$$
consequently
$$\abs{1-\omega^{2j+1}}=2 \sin\left(\frac{(2j+1)\pi}{2k}\right)\quad\hbox{ and }\quad\Arg(1-\omega^{2j+1})=\frac{(2j+1-k)\pi}{2k}.$$

Therefore,
$$U_k=\frac{k\pi^2}{12}-\frac12\sum_{j=0}^{k-1}\log^2\left(2\sin\frac{(2j+1)\pi}{2k}\right)
-\frac{\pi^2}{2}\sum_{j=0}^{k-1}\left(\frac{2j+1-k}{2k}\right)^2,$$

But,

\begin{align*}
\sum_{j=0}^{k-1}\left(\frac{2j+1-k}{2k}\right)^2
&=\sum_{j=0}^{k-1}\left(\frac{j^2}{k^2}-\frac{(k-1)j}{k^2}+\frac{(k-1)^2}{4k^2}\right)\\
&=\frac{(k-1)(2k-1)}{6k}-\frac{(k-1)^2}{4k} =\frac{k^2-1}{12k}, 
\end{align*}

hence,
$$U_k=\frac{(k^2+1)\pi^2}{24k}-\frac12\sum_{j=0}^{k-1}\log^2\left(2\sin\frac{(2j+1)\pi}{2k}\right),$$

which is the desired conclusion.
\end{proof}

\smallskip

{\bf Examples.} In particular,
\begin{align*}
\sum_{n=1}^\infty(-1)^{n-1}\frac{H_n}{n}& =\frac{\pi^2}{12}-\frac{1}{2}\log^22,\\
\sum_{n=1}^\infty(-1)^{n-1}\frac{H_{2n}}{n}& =\frac{5\pi^2}{48}-\frac{1}{4}\log^22,\\
\sum_{n=1}^\infty(-1)^{n-1}\frac{H_{3n}}{n}& =\frac{5\pi^2}{36}-\frac{1}{2}\log^22,\\
\sum_{n=1}^\infty(-1)^{n-1}\frac{H_{4n}}{n}& =\frac{17\pi^2}{96}-\frac{1}{8}\log^22-\frac{1}{2}\log^2(1+\sqrt2),\\
\sum_{n=1}^\infty(-1)^{n-1}\frac{H_{5n}}{n}& =\frac{13\pi^2}{60}-\frac{1}{2}\log^22-2\log^2\left(\frac{1+\sqrt5}{2}\right),\\
\sum_{n=1}^\infty(-1)^{n-1}\frac{H_{6n}}{n}& =\frac{37\pi^2}{144}-\frac{1}{4}\log^22-\frac12\log^2(2+\sqrt3).
\end{align*}

\smallskip

{\bf Conclusion.} In this paper, we have determined the sum of several families of numerical series related
 to harmonic numbers using very simple techniques from classical and complex analysis. We think
 that some of these results and techniques are important in their own right.

\end{document}